\documentclass[10pt]{elsarticle}
\usepackage[cp1251]{inputenc}
\usepackage[english]{babel}
\usepackage{amsmath}
\usepackage{amssymb}
\usepackage{amsfonts}
\usepackage{mathtools}
\usepackage{dsfont}
\usepackage{rotating}
\usepackage{graphicx}
\usepackage{floatflt,epsfig}
\usepackage{lineno,hyperref}
\usepackage{enumerate}
\usepackage{colortbl}
\usepackage{array,tabularx,tabulary,booktabs}
\usepackage{longtable}
\usepackage{multirow}
\usepackage{wrapfig}
\usepackage{subcaption}
\usepackage[table]{xcolor}
\newcolumntype{^}{>{\currentrowstyle}}

\usepackage{amsthm}
\usepackage[left=1in,right=1in,top=1in,bottom=1in]{geometry}

\journal{Arxiv}
\newtheorem{lemma}{Lemma}

\newtheorem{theorem}{Theorem}
\newtheorem{corollary}{Corollary}
\newtheorem{proposition}{Proposition}
\newtheorem{problem}{Problem}

\newcommand{\PG}[2]{\mathsf{PG}_{#1}(#2)}

\newcommand{\Or}[3][]{\mathsf{O}^{#1}_{#2}(#3)}
\newcommand{\Sp}[2]{\mathsf{Sp}_{#1}(#2)}
\newcommand{\Un}[2]{\mathsf{U}_{#1}(#2)}
\newcommand{\VO}[3][]{{VO}^{#1}_{#2}(#3)}

\newcommand{\Quadric}[1]{\mathcal{Q}(#1)}
\newcommand{\Aff}[2][]{\text{Aff}^{#1}(#2)}
\newcommand{\Proj}[1]{\text{Proj}(#1)}

\DeclarePairedDelimiter{\lrangle}{\langle}{\rangle}

\begin{document}
\renewcommand{\abstractname}{Abstract}
\renewcommand{\refname}{References}
\renewcommand{\tablename}{Figure.}
\renewcommand{\arraystretch}{0.9}
\thispagestyle{empty}
\sloppy

\begin{frontmatter}
\title{Tightness of the weight-distribution bound for strongly regular polar graphs}

\author[01]{Rhys J. Evans}
\ead{rhys.evans@imfm.si}

\author[02]{Sergey Goryainov}
\ead{sergey.goryainov3@gmail.com}

\author[03]{Leonid Shalaginov}
\ead{44sh@mail.ru}

\address[01] {Institute of Mathematics, Physics and Mechanics,\\ Jadranska ulica 19, 1000 Ljubljana, Slovenia}
\address[02] {School of Mathematical Sciences, Hebei International Joint Research Center for Mathematics and Interdisciplinary Science, Hebei Key Laboratory of Computational Mathematics and Applications, Hebei Workstation for Foreign Academicians,\\Hebei Normal University, Shijiazhuang  050024, P.R. China}
\address[03] {Chelyabinsk State University, Brat'ev Kashirinyh st. 129\\Chelyabinsk  454021, Russia}


\begin{abstract}
In this paper we show the tightness of the weight-distribution bound for the positive non-principle eigenvalue of strongly regular (affine) polar graphs and characterise the optimal eigenfunctions. Additionally, we show the tightness of the weight-distribution bound for the negative non-principle eigenvalue of some unitary polar graphs.
\end{abstract}

\begin{keyword}
classical polar space; strongly regular graph; weight-distribution bound 
\vspace{\baselineskip}
\MSC[2010] 05C25\sep 05E30\sep 51E20
\end{keyword}
\end{frontmatter}

\section{Introduction}

Recently, for a variety of distance-regular graphs, the eigenfunctions having the miminum cardinality of support were studied. These studies were initiated in \cite{KMP16}, surveyed in \cite{SV21} and further extended in \cite{GSY23, GY24, GP24, DGHS24}.
One of the main tools in these studies is the weight-distribution bound, a lower bound for the cardinality of support of an eigenfunction of a distance-regular graph. In particular, the tightness of the weight-distribution bound was shown in \cite{GY24} for both non-principal eigenvalues of the affine polar graphs $VO^+(4,q)$. Recently, in connection with eigenfunctions of symplectic graphs $Sp(4,q)$ whose cardinality of support meets the weight-distribution bound, a new infinite family of divisible design graphs was constructed \cite{DGHS24}.
Except for the motivation described in \cite{SV21}, the eigenfunctions whose cardinality of support meets the weight-distribution bound are of interest since they give a restriction \cite[Corollary 1]{GP24} on the equitable 2-partitions of the graphs.
Motivated by the results on strongly regular polar graphs, we initiate the studies of optimal eigenfunctions in strongly regular (affine) polar graphs.  

The main results of the paper are as follows. We first show the tightness of the weight-distribution bound for the positive non-principle eigenvalue of strongly regular (affine) polar graphs and characterise the optimal eigenfunctions. 

\begin{theorem}\label{thm:Main1}
Let $X$ be a strongly regular (affine) polar graph. Then the following statements hold.\\
{\rm (1)} Let $C_0,C_1$ be two distinct Delsarte cliques in $X$ such that the size of the intersection of $C_1 \cap C_2$ is maximum possible. Let $T_0 = C_0 \setminus C$ and $T_1 = C_1 \setminus C$, where $C = C_0 \cap C_1$.
Then the function $f:V(X) \mapsto \mathbb{R}$ taking value 1 on the vertices from $T_0$, value $-1$ on the vertices from $T_1$, and value 0 otherwise is an eigenfunction of $X$ corresponding to the positive non-principal eigenvalue, with the support meeting the weight-distribution bound.\\
{\rm (2)} Let $g$ be an eigenfunction of $X$ corresponding to the positive non-principal eigenvalue, with the support meeting the weight-distribution bound. Then $g = cf$ for some eigenfunction $f$ from item {\rm (1)} and a real number $c$.   
\end{theorem}

Additionally, we show the tightness of the weight-distribution bound for the negative non-principle eigenvalue of some unitary polar graphs.
\begin{proposition}\label{prop:Main2}
Let $q$ be a prime power, a square.
The weight-distribution bound is tight for the negative non-principal eigenvalue $\theta_2 = -(\sqrt{q}+1)$ of the unitary graph $U(4,q)$.     
\end{proposition}
Except for the optimal eigenfunctions of affine polar graphs $VO^+(4,q)$ and symplectic $Sp(4,q)$, constructed, respectively, in \cite{GY24} and \cite{DGHS24}, and the optimal eigenfunctions of unitary graphs $U(4,q)$, constructed in Proposition \ref{prop:Main2}, we do not know any examples of the tightness of the weight-distribution bound for the negative non-principle eigenvalue of strongly regular (affine) polar graphs. The following problem is thus of interest.
\begin{problem}
What are the eigenfunctions of strongly regular (affine) polar graphs corresponding to the negative non-principal eigenvalue and having the minimum cardinality of support?     
\end{problem}

The paper is organised as follows. In Section \ref{Preliminaries}, we give preliminary definitions and results. In Section \ref{CharacterisationPositiveEigenvalue}, we prove a number of statements that imply Theorem  \ref{thm:Main1}. In Section \ref{TightnessU4q}, we prove Proposition \ref{prop:Main2}.

\section{Preliminaries}\label{Preliminaries}
In this section we list some preliminary definitions and results.
\subsection{Strongly regular graphs}
A $k$-regular graph on $v$ vertices is called \emph{strongly regular} with parameters $(v,k,\lambda,\mu)$ if any two adjacent vertices have $\lambda$ common neigbours and any two distinct non-adjacent vertices have $\mu$ common neighbours. 
 A strongly regular graph $X$ is \emph{primitive} if both $X$ and its complement
are connected. 

\begin{lemma}[{\cite[Theorem 5.2.1]{GM15}}]\label{EigenvaluesSRG}
If $X$ is a primitive strongly regular graph with parameters $(v,k,\lambda,\mu)$, then $X$ has exactly 3 distinct eigenvalues $k,\theta_1,\theta_2$, such that $k>\theta_1>0>\theta_2$. The eigenvalues $\theta_1,\theta_2$ and their multiplicities can be derived from the parameters of $X$.
\end{lemma}

A clique $C$ in a graph is called \emph{regular} if every vertex that is not in $C$ has the same positive number of neighbors in $C$. The following lemma gives an upper bound on the clique number of a strongly regular graph, and shows that a maximum clique is regular if and only if its size agrees with the given upper bound.

\begin{lemma} [Delsarte-Hoffman bound, {\cite[Proposition 1.3.2]{BCN89}}]\label{HoffmanBound}
Suppose that $X$ is a strongly regular graph with parameters $(v,k,\lambda,\mu)$ and smallest eigenvalue $-m$.
Let $C$ be a clique in $X$. Then $|C| \le 1+\frac{k}{m}$, with equality if and only if every vertex that is not in $C$ has the same number of neighbors (namely $\frac{\mu}{m}$) in $C$. 
\end{lemma}

A clique in strongly regular graph whose size meets the Delsarte-Hoffman bound is called a \emph{Delsarte clique}.

\subsection{Weight-distribution bound for strongly regular graphs}

Let $\theta$ be an eigenvalue of a graph $X$. A real-valued function on the vertex set of $X$, $f$, is called a \emph{$\theta$-eigenfunction} of $X$ if it has
at least one non-zero value, and
for any vertex $\gamma$ in $X$ the condition
\begin{equation}\label{LocalCondition}
\theta\cdot f(\gamma)=\sum_{\substack{\delta\in{X(\gamma)}}}f(\delta)
\end{equation}
holds, where $X(\gamma)$ is the set of neighbours of the vertex $\gamma$. Then, the \emph{support} $f$ is the set of vertices of $X$ on which $f$ takes a non-zero value.

The following lemma gives a lower bound for the number of non-zeroes for an eigenfunction of a strongly regular graph. This bound is presented in \cite[Corollary 1]{KMP16} for distance-regular graphs, and we take the case for which the bound applies to strongly regular graphs.

\begin{lemma}\label{WDB}
Let $X$ be a primitive strongly regular graph with parameters $(v,k,\lambda,\mu)$ and let $\theta$ be a non-principal eigenvalue of $X$. Then a  $\theta$-eigenfunction of $X$ has at least $$1+|\theta|+|\frac{(\theta-\lambda)\theta-k}{\mu}|$$ non-zeroes, which is equal to $2(\theta_1+1)$ if $\theta = \theta_1$ and is equal to $-2\theta_2$ if $\theta = \theta_2$.
\end{lemma}

Lemma \ref{OptimalEigenfunctionsSRG} follows from \cite[Theorem 3, Theorem 4]{KMP16} and the fact that taking the complement of a strongly regular graph preserves the eigenspaces corresponding to the non-principal eigenvalues, which gives a description of the support of an eigenfunction of a strongly regular graph that meets the weight-distribution bound.
\begin{lemma}\label{OptimalEigenfunctionsSRG}
Let $X$ be a primitive strongly regular graph with eigenvalues $k>\theta_1 > 0 > \theta_2$. Then the following statements hold.\\
{\rm (1)} For a $\theta_2$-eigenfunction $f$, if the cardinality of support of $f$ meets the weight-distribution bound, then there exists an induced complete bipartite subgraph in $X$ with parts $T_0$ and $T_1$ of size $-\theta_2$. Moreover, up to multiplication by a constant, $f$ has value 1 on the vertices of $T_0$ and value $-1$ on the vertices of $T_1$. \\
{\rm (2)} For a $\theta_1$-eigenfunction $f$, if the cardinality of support of $f$ meets the weight-distribution bound, then there exists an induced pair of isolated cliques $T_0$ and $T_1$ in $X$ of size $-\overline{\theta_2} = -(-1-\theta_1) = 1+\theta_1$. Moreover, up to multiplication by a constant, $f$ has value 1 on the vertices of $T_0$ and value $-1$ on the vertices of $T_1$.
\end{lemma}

In view of Lemma \ref{OptimalEigenfunctionsSRG}, to show the tightness of the weight-distribution bound for non-principal eigenvalues it suffices to find a certain induced subgraph (a pair of isolated cliques $T_0$ and $T_1$ or a complete bipartite graph with parts $T_0$ and $T_1$) and show that each vertex outside of $T_0 \cup T_1$ has the same number of neighbours in $T_0$ and $T_1$.

\subsection{Polar spaces}

In this section, we will introduce polar spaces and present some basic results, most of which can be found in \cite{D16}. For further reference, see \cite{BS74,C15,T74,V59}.

Let $d$ be a positive integer, $q$ be a prime power, and $V=\mathbb{F}^d_q$. The \emph{Desarguesian projective space} $\PG{d-1}{q}$ is the point-line geometry with point set consisting of the 1-dimensional subspaces of $V$, line set consisting of the 2-dimensional subspaces of $V$, and incidence defined by containment.

Let $n$ be a positive integer. A \emph{(Veldkamp-Tits) polar space} of \emph{rank} $n$ is a pair $\Pi=(\mathcal{P},\Sigma)$, where the elements of $\mathcal{P}$ are called the \emph{points} of $\Pi$, the elements of $\Sigma$ is a set of subsets of $\mathcal{P}$ called \emph{singular subspaces} of $\Pi$, and the following hold:

\begin{enumerate}[(I)]
    \item\label{AxiomI} For all $L\in \Sigma$, the points and singular subspaces contained in $L$ form a projective space of dimension $d\in\{-1,0,\dots,n-1\}$. We will call $d$ the \emph{dimension} of $L$, and denote it by $\dim(L)$.
    \item\label{AxiomII} For all $L,M\in \Sigma$, $L\cap M\in \Sigma$.
    \item\label{AxiomIII} For all $L\in \Sigma$ such that $\dim(L)=n-1$ and point $p\in \mathcal{P}\setminus L$, there exists a unique singular subspace $M\in \Sigma$ such that $p\in M$ and $\dim(M\cap L)=n-2$. In this case, $L\cap M$ consists of the points of $L$ which are contained together with $p$ in some singular subspace of dimension $1$.
    \item\label{AxiomIV} There exists $L,M\in \Sigma$ such that $L\cap M=\emptyset$ and $\dim(L)=\dim(M)=n-1$.
\end{enumerate}

\begin{lemma}\label{lem:polarProperties}
    Let $\Pi=(\mathcal{P},\Sigma)$ be a polar space of rank $n$. Then:
    \begin{enumerate}
        \item\label{lem:polarProperties-maximalDimension} All maximal singular subspaces have dimension $n-1$;
        \item\label{lem:polarProperties-collinearPoints} Any set $X$ of pairwise collinear points is contained in a maximal singular subspace;
        \item For any singular subspace $L$, there are maximal singular subspaces $M_1,M_2$, such that $L=M_1 \cap M_2$.
    \end{enumerate}
\end{lemma}
\begin{proof}
    See the proof of \cite[Theorem 7.7(a)]{C15}, \cite[Theorem 7.3]{D16}, and \cite[Theorem 7.12]{D16} repsectively.
\end{proof}
~\\

For sets $X_i$ such that the union consists of pairwise collinear points, denote by $\left[X_1,X_2\dots\right]$ the smallest singular subspace containing the points of the union of the sets $X_i$.  

Let $d$ be a positive integer and $q$ be a prime power. A polar space $\Pi=(\mathcal{P},\Sigma)$ is an \emph{embedded} polar space if $\mathcal{P}\subseteq \PG{d-1}{q}$ for some positive integer $d$. The \emph{order} of an embedded polar space $\Pi$ of rank $n$ is the pair $(q,t)$, where $t+1$ is this number of maximal singular subspaces which contain a given singular subspace of dimension $n-2$. A proof that $t$ is well-defined and positive can be found in \cite[Section 2.2.5]{BV22}. 

The embedded polar spaces of rank $n\geqslant 2$ have been fully classified. In Table \ref{tab:polarInfo} we list the embedded polar spaces of rank $n\geqslant 2$. For each polar space, the table contains the notation for the space, the dimension of the ambient vector space, and the order of the space. The last column contains a classical parameter $e$ for each of these spaces, which will be used later in counting arguments.

\begin{table}[ht]
\centering
\begin{tabular}{l|c|c|c|c}
  Name & Notation & dim(V) & Order & $e$\\
  \hline
  Symplectic & $\Sp{2n}{q}$ & $2n$ & $(q,q)$ & 1\\
  \hline
  Hyperbolic orthogonal & $\Or[+]{2n}{q}$ & $2n$ & $(q,1)$ &  0\\
  \hline
  Parabolic orthogonal & $\Or{2n+1}{q}$ & $2n+1$ & $(q,q)$ & 1\\
  \hline
  Elliptic orthogonal & $\Or[-]{2n+2}{q}$ & $2n+2$ & $(q,q^2)$ & 2\\
  \hline
  Small unitary & $\Un{2n}{\sqrt{q}}$ & $2n$ & $(q,q^{1/2})$ & 1/2\\
  \hline
  Large unitary & $\Un{2n+1}{\sqrt{q}}$ & $2n+1$ & $(q,q^{3/2})$ & 3/2\\
\end{tabular}
\caption{Details of the embedded polar spaces of rank $n$.}\label{tab:polarInfo}
\end{table}

In particular, we will be interested in the spaces $\Or[+]{d}{q}$ and $\Or[-]{d}{q}$. Let $V$ be a vector space over $\mathbb{F}_q$ of dimension $2m$, for positive integer $m$. Further, let $Q$ be a nondegenerate quadratic form on $V$ of type $\epsilon\in \{+1,-1\}$ (i.e. hyperbolic and elliptic respectively). In the remaining, we will identify $\epsilon$ with its sign when it is convenient for notational purposes. 

The polar space $\Or[\epsilon]{d}{q}$ has singular subspaces consisting of subspaces $W\subseteq V$ such that $Q(w)=0$ for all $w\in W$, and points consisting of the 0-dimensional singular subspaces.
  
\subsection{Polar and affine polar graphs}

Let $\Pi=(\mathcal{P},\Sigma)$ be a polar space of rank $n$. The \emph{collinearity graph} of $\Pi$, $\Gamma(\Pi)$, is the graph with vertex-set $\mathcal{P}$, and for which distinct vertices $x,y$ are adjacent if and only if $x,y\in L$ for some $L\in \Sigma$. 

\begin{lemma}\label{lem:polarGraphParameters}
    Let $\Pi$ be an embedded polar space of rank n and order $(q,t)$. Then $\Gamma(\Pi)$ is strongly regular with eigenvalues $\theta_1 = q^{n-1}-1$ and $\theta_2 = -tq^{n-2}-1$.

    Furthermore, a clique in $\Gamma(\Pi)$ with order meeting the Delsarte-Hoffman bound has $(q^n-1)/(q-1)$ vertices and is regular with nexus $(q^{n-1}-1)/(q-1)$.
\end{lemma}
\begin{proof}
    The strongly regular graph parameters of $\Gamma(\Pi)$ are derived in \cite[Theorem 2.2.12]{BV22}. The rest follows from Lemmas \ref{EigenvaluesSRG} and \ref{HoffmanBound}.
\end{proof}

~\\

Let $V$ be a vector space over $\mathbb{F}_q$ of dimension $2m$, for positive integer $m$. Further, let $Q$ be a nondegenerate quadratic form on $V$ of type $\epsilon\in \{+1,-1\}$. The \emph{polarisation} of $Q$, $B$, is the bilinear form such that $B(x,y)=Q(x+y)-Q(x)-Q(y)$ for all $x,y\in V$.

The \emph{affine polar graph} $\VO[\epsilon]{2m}{q}$ is the graph with vertex-set the elements of $V$, and for which distinct vertices $x,y$ are adjacent if and only if $Q(x-y)=0$. 

\begin{lemma}\label{lem:affine_parameters}
    Let $V$ be a vector space over $\mathbb{F}_q$ of dimension $d=2m-\epsilon+1$, for positive integer $m$. Further, let $Q$ be a nondegenerate quadratic form on $V$ of type $\epsilon\in \{+1,-1\}$. Then the affine polar graph $\VO[\epsilon]{2m}{q}$ is strongly regular with eigenvalues $\theta_1 = \epsilon(q-1)q^{m-1}-1$ and $\theta_2 = -\epsilon q^{m-1}-1.$    
 
    Furthermore, a clique in $\VO[\epsilon]{2m}{q}$ with order meeting the Delsarte-Hoffman bound has $1+(q^m-\epsilon)(q^{m-1}+\epsilon)/(\epsilon q^{m-1}+1)$ vertices and is regular with nexus $q^{m-1}(q^{m-1}+\epsilon)/(\epsilon q^{m-1}+1).$
\end{lemma}    
~\\

\section{Characterisation of optimal $\theta_1$-eigenfunctions}\label{CharacterisationPositiveEigenvalue}
In this section we will characterise the optimal eigenfunctions for the positive non-principal eigenvalue $\theta_1$ of each graph defined in Sections \ref{Preliminaries}. The constructions and proofs in Sections \ref{ssec:theta1_polar} and \ref{sssec:theta1_hyperbolic}  are closely related, while the construction in Section \ref{sssec:theta1_elliptic} is slightly different, but both constructions use maximal singular subspaces as there basis. We will use the properties of polar spaces to count exactly how many such eignefunctions these graphs have.

\subsection{The $\theta_1$-eigenfunctions in polar graphs}\label{ssec:theta1_polar}

In this section, we will characterise the optimal $\theta_1$-eigenfunctions of the polar graphs. We start by defining some notation for certain subsets of the singular subspaces of a given polar space, which we will find useful.

Let $\Pi=(\mathcal{P},\Sigma)$ be a polar space. For any $L\in\Sigma$, we define
\begin{equation*}
    \Sigma_L = \left\{M\in \Sigma: L\subsetneq M,\dim(M)=n-1\right\}.
\end{equation*}
Note that $\Sigma_L$ is a set of maximal singular subspaces of $\Pi$. If $\Pi$ is an embedded polar space with order $(q,t)$, then for any $L\in\Sigma$ such that $\dim(L)=n-2$, we have $|\Sigma_L|=t+1\ge 2$. For each $L\in \Sigma$, we define
\begin{equation*}
    \Delta_L  = \left\{\{M\setminus L, N\setminus L\}: M,N\in \Sigma_L,M\neq N\right\}.
\end{equation*}

We show that when $\dim(L)=n-2$, each element of $\Delta_L$ define an optimal $\theta_1$-eigenfunction.

\begin{lemma}\label{lem:notcollinear}
    Let $\Pi=(\mathcal{P},\Sigma)$ be an embedded polar space of rank $n$ and order $(q,t)$, $L$ be a singular subspace of $\Pi$ of dimension $n-2$, and $M,N\in \Sigma_L$ be distinct. Then:
    \begin{enumerate}
        \item\label{lem:notcollinear-collinear} for any $x,y\in M\setminus L$, $x$ and $y$ are collinear; 
        \item\label{lem:notcollinear-notcollinear} for any $x\in M\setminus L,y\in N\setminus L$, $x$ and $y$ are not collinear;
        \item\label{lem:notcollinear-eigenfunction} the function $f:\mathcal{P}\to \mathbb{R}$, such that
        \begin{equation*}
            f(z) = 
                \left\{
                  \begin{array}{cc}
                    1, & z\in M \setminus L; \\
                    -1, & z\in N \setminus L; \\
                    0, & \text{otherwise}.
                  \end{array}
                \right.
        \end{equation*}
        is a $\theta_1$-eigenfuction of $\Gamma(\Pi)$.
    \end{enumerate}

\end{lemma}
\begin{proof}
    1. This follows immediately from Axiom \eqref{AxiomI} and $x,y\in M$.
    
    2. Suppose otherwise. Then as $\left[ y, L\right]=N$ and $x\notin N$, we have a chain of singular subspaces $L\subsetneq \left[ y, L\right] \subsetneq\left[ x, y, L\right]$. By Lemma \ref{lem:polarProperties} \ref{lem:polarProperties-maximalDimension}, this contradicts maximality of $N$.

    3. First note that $\theta_1=q^{n-1}-1=|M\setminus L|-1=|N\setminus L|-1$. As all elements contained in a singular subspace are collinear, condition \eqref{LocalCondition} is satisfied for $z\in L$. By parts \ref{lem:notcollinear-collinear} and \ref{lem:notcollinear-notcollinear}, we see that the condition  \eqref{LocalCondition} is satisfied for $z\in (M\cup N)\setminus L$.
    
    Consider $z\notin (M\cup N)$. Note that as  $M,N$ have dimension $n-1$, they contain $(q^n-1)/(q-1)$ points. Therefore $M$ and $N$ are regular cliques in $\Gamma(\Pi)$ by Axiom \eqref{AxiomI} and Lemma \ref{lem:polarGraphParameters}. In particular, $z$ is adjacent to the same number of vertices in $M\setminus L$ and $N\setminus L$, and thus condition \eqref{LocalCondition} is satisfied for $z$. The result follows.
    
\end{proof}

Next, we show that any pair of isolated cliques of the sizes given in the above example must come from an element of $\Delta_L$ for some singular subspace $L$ of dimension $n-2$.

\begin{proposition}\label{prop:polarCharacterisation}
    Let $\Pi=(\mathcal{P},\Sigma)$ be an embedded polar space of rank $n\geqslant 2$ with order $(q,t)$. Then $T=\{T_0,T_1\}$ is a pair of isolated cliques of size $\theta_1+1$ in $\Gamma(\Pi)$ if and only if $T\in \Delta_L$ for some $L\in \Sigma$ with $\dim(L)=n-2$.
\end{proposition}
\begin{proof}
    ($\implies$) Suppose $T=\{T_0,T_1\}$ is a pair of isolated cliques of size $\theta_1+1$ in $\Gamma(\Pi)$. Then we know $|T_0|=|T_1|=\theta_1+1=q^{n-1}$ by Lemma \ref{lem:polarGraphParameters}. As this is larger than the size of a singular subspace of dimension $n-2$, Lemma \ref{lem:polarProperties} \ref{lem:polarProperties-collinearPoints} implies that $M_i=\left[T_i\right]$ is the unique maximal clique containing $T_i$. 

    Consider $p\in T_0$. Note that $M_0\neq M_1$ and $p\notin M_1$ because $T_0$ and $T_1$ are isolated. By Axiom \eqref{AxiomIII}, there is a unique singular subspace $N_1$ such that $p\in N_1$ and $N_1\cap M_1$ has dimension $n-2$. But $|M_1\setminus T_1|=(q^{n-1}-1)/(q-1)$, which is the size of $N_1\cap M_1$. As $T_0$ and $T_1$ are isolated, $N_1\cap M_1\subseteq M_1\setminus T_1$, and by pigeonhole principle we have equality. 
    
    As this holds for all $p\in T_0$, we see that $T_0\cup M_1\setminus T_1$ is a maximal singular subspace. But by uniqueness of $M_0$, we have $M_0=T_0\cup (M_1\setminus T_1)$, and $M_0\cap M_1=M_0\setminus T_0= M_1\setminus T_1$. Also, $L=M_0\cap M_1$ is a singular subspace of size $|M_0\setminus T_0|=(q^{n-1}-1)/(q-1)$, showing that $\dim(L)=n-2$. We have shown that $T=\{M_0\setminus L,M_1\setminus L\}$, and $T\in \Delta_L$.

    ($\impliedby$) Let $T=\{M_0\setminus L, M_1\setminus L\}\in \Delta_L$ for singular subspace $L$ of dimension $n-2$, and let $T_0=M_0\setminus L,T_1=M_1\setminus L$. By Lemma \ref{lem:notcollinear}, the sets $T_0$ and $T_1$ are isolated cliques. By Axiom \eqref{AxiomI} we have $|M_i|=(q^n-1)/(q-1), |L|=(q^{n-1}-1)/(q-1)$, and $|M_i\setminus L|=q^{n-1}=\theta_1 + 1$. 
\end{proof}

The two results above gives a characterisation of optimal $\theta_1$-eigenfunctions of a polar graph: they are the difference of indicator functions $1_A - 1_B$, where $A,B\in \Delta_L$ for some $(n-2)$-dimensional singular subspace. Now we give count the number of such distinct functions by calculating the size of the sets $\Delta_L$ and their intersections.

\begin{lemma}\label{lem:SigmaDelta}
    Let $\Pi=(\mathcal{P},\Sigma)$ be an embedded polar space of rank $n\geqslant 2$ and order $(q,t)$. Further let $L_0,L_1\in \Sigma$ be distinct, with $\dim(L_0)=\dim(L_1)=n-2$. Then:
    \begin{enumerate}
        \item\label{lem:SigmaDelta-Sigma} $|\Sigma_{L_0}\cap\Sigma_{L_1}|\leqslant 1$; 
        \item\label{lem:SigmaDelta-Delta}$\Delta_{L_0}\cap\Delta_{L_1}=\emptyset$;
        \item\label{lem:SigmaDelta-Size} $|\Delta_{L_0}|=\binom{t+1}{2}$. 
    \end{enumerate}
\end{lemma}
\begin{proof}
    1. Suppose $M,N\in \Sigma_{L_0}\cap\Sigma_{L_1}$. Then $M\cap N$ is a singular subspace by Axiom \eqref{AxiomII}, of dimension at most $n-2$, contradicting the assumption $L_0\cup L_1\subseteq M\cap N$.

    2. Suppose there are distinct $M_0,N_0\in \Sigma_{L_0}$ and distinct $M_1,N_1\in \Sigma_{L_1}$ such that $\{M_0\setminus L_0, N_0\setminus L_0\}=\{M_1\setminus L_1, N_1\setminus L_1\}$. Without loss of generality, assume $M_0\setminus L_0=M_1\setminus L_1$ and $N_0\setminus L_0=N_1\setminus L_1$.

    Note that $M_0\setminus L_0$ is a set of collinear points, and $|M_0\setminus L_0|=q^{n-1}$ is larger than a projective space of dimension $n-2$. Therefore, $\left[M_0\setminus L_0\right]$ is the unique maximal singular subspace containing $M_0\setminus L_0$ by Lemma \ref{lem:polarProperties} \ref{lem:polarProperties-collinearPoints}. Also, $M_0$ and $M_1$ are maximal singular subspaces containing $M_0\setminus L_0$, so $M_0=M_1$. Similarly, we can see that $N_0=N_1$. But this means $M_0,N_0\in \Sigma_{L_0}\cap\Sigma_{L_1}$, contradicting part \ref{lem:SigmaDelta-Sigma}.

    3. From the proof of part \ref{lem:SigmaDelta-Delta}, we see that any elemnt of $\Delta_{L_0}$ is defined by a unique pair of elements of $\Sigma_{L_0}$. The result follows by noting that $|\Sigma_{L_0}|=t+1$ by definition. 
\end{proof}

Now that we know that the sets $\Delta_L$ are disjoint, we can use well-known results involving the counting of singular subspaces of polar spaces to count how many optimal $\theta_1$-eigenfunctions a polar graph has. Here the classical parameter $e$ from Table \ref{tab:polarInfo} appears.

\begin{corollary}
    Let $\Pi=(\mathcal{P},\Sigma)$ be an embedded polar space of rank $n\geqslant 2$ found in Table \ref{tab:polarInfo}, with order $(q,t)$ and parameter $e$. Then there are exactly 
    \begin{equation*}
        \binom{t+1}{2}\left(\frac{q^n-1}{q-1}\right)\prod_{i=0}^{n-2}\left(q^{n+e-i-1}+1\right)
    \end{equation*}
    pairs $\{T_0,T_1\}$ of isolated cliques of size $\theta_1+1$
\end{corollary}
\begin{proof}
    Let $N$ be the number of such pairs of isolated cliques. By Proposition \ref{prop:polarCharacterisation}, these pairs are exactly the elements of $\Delta_L$ for some singular subspace of $\Pi$ with $\dim(L)=n-2$. By Lemma \ref{lem:SigmaDelta} parts \ref{lem:SigmaDelta-Delta} and \ref{lem:SigmaDelta-Size}, we have $\delta=|\Delta_L|=t(t+1)/2$ for all such $L$ and $N/\delta$ is the number of singular subspaces of $\Pi$ of dimension $n-2$. The result follows by \cite[Lemma 9.4.1]{BCN89}.
\end{proof}

\subsection{The affine polar graphs}\label{ssec:theta1_affine}

In this section we characterise and count the optimal $\theta_1$-eigenfunctions for the affine polar graphs. To do this, we introduce notation which will help to transfer our knowledge of the associated polar spaces to the ambient vector space.

Let $V$ be a vector space over $\mathbb{F}_q$ of dimension $2m$, $m\geqslant 1$, and $Q$ be the nondegenerate quadratic form on $V$ of type $\epsilon\in \{+1,-1\}$. The \emph{quadric} $\Quadric{V}$ is the set $\Quadric{V}=\{v\in V:Q(v)=0\}$. The affine polar graphs $\VO[\epsilon]{2m}{q}$ are closely related to the orthogonal polar graphs $\Or[\epsilon]{2m}{q}$. For any subset of points $P$ of $\Or[\epsilon]{2m}{q}$, we define 
\begin{align*}
    &\Aff{P}  = \{v\in V:v\in p\text{ for some }p\in P\},\text{ and}\\
    &\Aff[*]{P}  = \Aff{P}\setminus\{0\}.
\end{align*} 
Note that $\Aff{P}\subseteq \Quadric{V}$ for any set of points $P$, and $\Aff{P}$ consists of the elements of a vector space if and only if $P$ is the points of a projective space. 

For any subset $U\subseteq V$, we define
\begin{equation*}
    \Proj{U}=\{\lrangle{v}:v\in U,v\neq 0\}.
\end{equation*} 
Note that $\Proj{U}$ is a set of points of $\Or[\epsilon]{2m}{q}$ if and only if $U\subseteq \Quadric{V}$.

For a singular subspace $L$ of $\Or[\epsilon]{2m}{q}$, we define
\begin{equation*}
    V\Delta_L = \left\{\{\Aff[*]{M\setminus L},\Aff[*]{N\setminus L}\}:M,N\in \Sigma_L,M\neq N\right\}.
\end{equation*}

Now we start to investigate the structure of the affine polar graphs and their cliques.

\begin{lemma}\label{lem:polarToAffineCliques}
     Let $V$ be a vector space over $\mathbb{F}_q$ of dimension $2m$, $m\geqslant 1$, and $Q$ be the nondegenerate quadratic form on $V$ of type $\epsilon\in \{+1,-1\}$. We have:
     \begin{enumerate}
        \item\label{lem:polarToAffineCliques-shift} for all $u,v,w\in V$, $u$ and $w$ are adjacent in $\VO[\epsilon]{2m}{q}$ if and only if $v+u$ is adjacent to $v+w$ (i.e. the function $\phi_v:V\to V,\phi_v(u)=v+u$ is an automorphism of $\VO[\epsilon]{2m}{q}$).
        \item\label{lem:polarToAffineCliques-spaces} For a singular subspace $L$ of $\Or[\epsilon]{2m}{q}$, $v+\Aff{L}$ is a clique in $\VO[\epsilon]{2m}{q}$.
     \end{enumerate}
\end{lemma}
\begin{proof}
    1. Vertices $u,w$ are adjacent if and only if $Q(u-w)=0$. Then we have $Q((u+v)-(w+v))=Q(u-w)=0$.

    2. We may assume $v=0$ by part \ref{lem:polarToAffineCliques-shift}. Let $u,w\in \Aff{L}$ be distinct. If $u=0$, as $\Aff{L}\subseteq \Quadric{V}$ we have $Q(w-u)=Q(w)=0$. 
    
    Now assume $u,w\in \Aff{L}$ are distinct and nonzero. Then $p=\lrangle{u},r=\lrangle{w}$ are points in $L$, and are collinear in $\Or[\epsilon]{2m}{q}$ as $L$ is a singular subspace. This means $Q$ is identically zero on the vector space $\lrangle{u,w}$. In particular, $Q(u-w)=0$.
\end{proof}

The next Lemma gives a characterisation of maximal cliques in the affine polar graphs, which will be useful later.
    
\begin{lemma}\label{lem:affineCliques}
     Let $V$ be a vector space over $\mathbb{F}_q$ of dimension $2m$, $m\geqslant 1$, and $Q$ be the nondegenerate quadratic form on $V$ of type $\epsilon\in \{+1,-1\}$. For distinct cliques $C$ and $D$ in $\VO[\epsilon]{2m}{q}$, we have;
     \begin{enumerate} 
        \item\label{lem:affineCliques-shift} $v+C$ is a clique for all $v\in V$;
        \item\label{lem:affineCliques-contained} for all $v\in C$, there is a (maximal) singular subspace $L$ of $\Or[\epsilon]{2m}{q}$ such that $C\subseteq v+\Aff{L}$;
        \item\label{lem:affineCliques-maximal} if $C$ is maximal, there exists a unique maximal singular subspace $L$ of $\Or[\epsilon]{2m}{q}$ such that $C=v+\Aff{L}$ for all $v\in C$;
        \item\label{lem:affineCliques-intersecting} if $C$ and $D$ are maximal cliques, either $C\cap D=\emptyset$ or there are maximal singular subspaces $M,N$ of $\Or[\epsilon]{2m}{q}$ such that $C=v+\Aff{M},D=v+\Aff{N}$ and $C\cap D= v+\Aff{M\cap N}$ for all $v\in C\cap D$.
     \end{enumerate}
\end{lemma}
\begin{proof}
    1. This follows immediately from Lemma \ref{lem:polarToAffineCliques} \ref{lem:polarToAffineCliques-shift}.
    
    2. Let $B$ be the polarisation of $Q$. By definition of adjacency in $\VO[\epsilon]{2m}{q}$, $v+C$ is a clique for all vectors $v\in V$. Let $v\in C$, and consider the clique $D=-v+C$. Then $0\in D$, so for all non-zero vectors $u,w\in D$ we must have $0=Q(u)=Q(w)=Q(w-u)$, which implies $B(u,w)=0$. But then for all $\alpha,\beta\in \mathbb{F}_q$, we have $Q(\alpha u+ \mu w)=B(\alpha u,\beta w) + Q(\alpha u) + Q(\beta w) = \alpha\beta B(u,w) + \alpha^2Q(u) + \beta^2Q(w)=0$. 

    We have proven that for all $u,w\in D$, $\lrangle{u,w}\subseteq \Quadric{V}$. Therefore, $\Proj{D}$ is a set of pairwise collinear points of $\Or[\epsilon]{2m}{q}$, which is contained in some (maximal) singular subspace $L$ of $\Or[\epsilon]{2m}{q}$ by Lemma \ref{lem:polarProperties} \ref{lem:polarProperties-collinearPoints}. By observing $D\subseteq \Aff{\Proj{D}}\subseteq\Aff{L}$, we see that $C\subseteq v+\Aff{L}$.

    3. By Lemma \ref{lem:polarToAffineCliques} and parts \ref{lem:affineCliques-shift} and \ref{lem:affineCliques-contained}, it follows that for any $v\in C$, $C=v+\Aff{L}$, where $L$ is a maximal singular subspaces. Suppose we have $u,w\in C$ and maximal singular subspaces $L_u,L_w$ such that $C=u+\Aff{L_u}=w+\Aff{L_w}$. As $\Aff{L_u}$ and $\Aff{L_w}$ are the elements of vector spaces of equal dimension, this forces $\Aff{L_u}=\Aff{L_w}$ and $L_u=L_w$.

    4. Suppose $v\in C\cap D$. By part \ref{lem:affineCliques-maximal} there exist unique maximal singular subspaces $M,N$ of $\Or[\epsilon]{2m}{q}$ such that $C=v+\Aff{M},D=v+\Aff{N}$. The result follows after observing $(-v+C)\cap (-v+D) = -v + (C\cap D)$ and $\Aff{M}\cap \Aff{N}=\Aff{M\cap N}$.
\end{proof}

~\\

\subsubsection{The $\theta_1$ eigenfunctions of the hyperbolic affine polar graphs}\label{sssec:theta1_hyperbolic}

In this section we consider the affine polar graphs corresponding to a quadratic form of type $+1$ (the hyperbolic case). For these graphs, we will show that the optimal $\theta_1$-eigenfunctions come from translations of the isolated cliques we have seen in Section \ref{ssec:theta1_polar}.

We begin by showing that any translation of the isolated cliques we have studied in the polar graph $\Gamma(\Or[+]{2m}{q})$ define an optimal $\theta_1$-eigenfunction in $\VO[+]{2m}{q}$.

\begin{lemma}\label{lem:hyperbolicConstruction}
    Let $V$ be a vector space over $\mathbb{F}_q$ of dimension $2m$, $m\geqslant 1$, and $Q$ be the nondegenerate quadratic form on $V$ of type $+1$. Further, let $v\in V$, $L$ be a singular subspace of $\Or[+]{2m}{q}$ of dimension $(m-2)$, and $M,N\in \Sigma_L$ be distinct. Then in the graph $\VO[+]{2m}{q}$:
    \begin{enumerate}
        \item\label{lem:hyperbolicConstruction-adj} for all distinct $x,y\in v + \Aff[*]{M\setminus L}$, $x$ and $y$ are adjacent;
        \item\label{lem:hyperbolicConstruction-notadj} for all $x\in v + \Aff[*]{M\setminus L},y\in v + \Aff[*]{N\setminus L}$, $x$ and $y$ are not adjacent;
        \item\label{lem:hyperbolicConstruction-func} the function $f:\mathcal{P}\to \mathbb{R}$, such that
        \begin{equation*}
            f(z) = 
                \left\{
                  \begin{array}{cc}
                    1, & z\in v + \Aff[*]{M \setminus L}; \\
                    -1, & z\in v + \Aff[*]{N \setminus L}; \\
                    0, & \text{otherwise}.
                  \end{array}
                \right.
        \end{equation*}
        satisfies condition  \eqref{LocalCondition} for $\theta_1=q^m-q^{m-1}-1$.
    \end{enumerate}
\end{lemma}
\begin{proof} 
    As being a adjacent, nonadjacent, and satisfying condition  \eqref{LocalCondition} is invariant under the action of an automorphism, we can assume $v=0$. Parts \ref{lem:hyperbolicConstruction-adj} and \ref{lem:hyperbolicConstruction-notadj} follow from Lemma \ref{lem:notcollinear} \ref{lem:notcollinear-collinear} and \ref{lem:notcollinear-notcollinear} respectively.

    3. Note that $\Aff{L}\cup \Aff[*]{M\setminus L}\cup \Aff[*]{N\setminus L} = \Aff{M}\cup \Aff{N} = \Aff{M\cup N}$. Therefore $|\Aff[*]{M\setminus L}|=|\Aff[*]{N\setminus L}|=q^m-q^{m-1}=\theta_1+1$. For $z\in \Aff{L}$ and $z\in \Aff[*]{M\setminus L}\cup \Aff[*]{N\setminus L}$ can be verified using parts \ref{lem:hyperbolicConstruction-adj} and \ref{lem:hyperbolicConstruction-notadj}.
        
    Let $z\notin \Aff{M\cup N}$ and $z$ have exactly the neighbours $M_z$ in $\Aff[*]{M\setminus L}$, $N_z$ in $\Aff[*]{M\setminus L}$ and $L_z$ in $\Aff{L}$. By Lemmas \ref{HoffmanBound} and \ref{lem:affine_parameters}, $\Aff{M}$ and $\Aff{N}$ are cliques with nexus $q^{m-1}$. Using the fact that $\Aff{M}\cap \Aff{N}=\Aff{L}$ and the above, we have $q^{m-1}=|M_z|+|L_z|=|N_z|+|L_z|$, showing that $|M_z|=|N_z|$. This shows that $f$ satisfies condition \eqref{LocalCondition} for $\theta_1$.
\end{proof}

Now we characterise the isolated cliques of the sizes we are interested in, which uses the characterisation in Proposition \ref{prop:polarCharacterisation} for polar graphs.

\begin{proposition}\label{prop:affineCharacterisation}
    Let $V$ be a vector space over $\mathbb{F}_q$  of dimension $2m$, $m\geqslant 1$, and $Q$ be the nondegenerate quadratic form on $V$ of type $+1$. For isolated cliques $T_0,T_1$ of size $\theta_1+1$ in $\VO[+]{2m}{q}$, we have cliques $C_0,C_1$ and maximal singular subspaces $M_0,M_1$ of $\Or[+]{2m}{q}$ such that;
    \begin{enumerate}
        \item $C_i$ are maximal cliques,  $C_i=v_i+\Aff{M_i}$ for all $v_i\in C_i$;
        \item $|C_0\cap C_1|=q^{m-1}$ and there exists a singular subspaces $L$ of $\Or[+]{2m}{q}$ of dimension $(n-2)$ such that $\{T_0,T_1\}\in v + V\Delta_L$ for all $v\in C_0\cap C_1$.
    \end{enumerate}
\end{proposition}
\begin{proof}
    Let $D_i$ be maximal cliques containing $T_i$. By Lemma \ref{lem:affineCliques} \ref{lem:affineCliques-intersecting}, $|D_i|=q^m$, and by Lemmas \ref{HoffmanBound} and \ref{lem:affine_parameters}, $D_i$ have nexus $q^{m-1}$.
    
    Suppose $D_0\cap D_1=\emptyset$.  For $u\in T_0$, as $T_0,T_1$ are isolated and $|D_1\setminus T_1|=q^{m-1}$, $u$ must be adjacent to all vertices in $D_1\setminus T_1$. This shows that each $w\in D_1\setminus T_1$ is adjacent to every element in $T_0$. Therefore $D_0$ has nexus at least $|T_0|=q^m-q^{m-1}\geqslant q^{m-1}$. Therefore $q=2$ and $(D_0\setminus T_0)\cup T_0,(D_0\setminus T_0)\cup T_1$ are both maximal cliques containing $T_0,T_1$ respectively.

    Therefore we can take maximal cliques $C_i$ which contain $T_i$ and $C_0\cap C_1\neq\emptyset$.    Let $v\in C_0\cap C_1$ and consider $S_0=\Proj{-v+T_0},S_1=\Proj{-v+T_1}$ and $S=\{S_0,S_1\}$. Then $S_0,S_1$ are isolated cliques of $\Gamma(\Or[+]{2m}{q})$, and as $v\notin T_i$, we have $|S_i|\geqslant (q^m-q^{m-1})/(q-1)=q^{m-1}$. By Proposition \ref{prop:polarCharacterisation}, any pair of subsets $R_0\subseteq S_0,R_1\subseteq S_1$ such that $|R_0|=|R_1|=q^{m-1}$ must be contained in unique maximal singular subspaces $M_0,M_1$ such that $L=M_0\cap M_1$ has dimension $m-2$. But this means $M_0,M_1$ is the unique maximal containing $S_0,S_1$ respectively. Then $|S_i|\leqslant |M_i\setminus L|=q^{m-1}$, proving that $|S_i|=q^{m-1}$, $S_i=M_i\setminus L$,  and $S_i$ consists of the nonzero scalar multiples of $q^{m-1}$ independent vectors, i.e. $-v+T_i=\Aff[*]{S_i}$. Therefore, $\{S_0,S_1\}\in \Delta_L$ and $\{T_0,T_1\}\in v+ V\Delta_L$.
\end{proof}
~\\
This finishes the characterisation of optimal $\theta_1$-eigenfunctions in hyperbolic affine polar graphs as those coming from translations of those coming from the hyperbolic polar space. Next we find when the translations of the sets $\Delta_L$ can intersect.

\begin{lemma}\label{lem:hyperbolicDistinct}
    Let $V$ be a vector space over $\mathbb{F}_q$  of dimension $2m$, $m\geqslant 1$, and $Q$ be the nondegenerate quadratic form on $V$ of type $+1$. Further let $u,w\in V$ and $L_0,L_1$ be $(m-2)$-dimensional singular subspaces of $\Or[+]{2m}{q}$. Then $(u+V\Delta_{L_0})\cap (w+V\Delta_{L_1})=\emptyset$ or $(u+V\Delta_{L_0})=(w+V\Delta_{L_1})$, with equality if and only if $L_0=L_1$ and $u-w\in \Aff{L_0}$.
\end{lemma}
\begin{proof}
    Suppose  $(u+V\Delta_{L_0})\cap (w+V\Delta_{L_1})\neq\emptyset$, so we have $M_i,N_i\in \Sigma_{L_i}$ such that $\Aff[*]{M_1\setminus L_1}=u-w +\Aff[*]{M_0\setminus L_0}$ and $\Aff[*]{N_1\setminus L_1}=u-w +\Aff[*]{N_0\setminus L_0}$. As $|\Aff[*]{M_i\setminus L_i}|=|\Aff[*]{N_i\setminus L_i}|=q^m-q^{m-1}$, we have $|M_i\setminus L_i|=|N_i\setminus L_i|=q^{m-1}$. Therefore, these sets are larger than an $(m-2)$-dimensional singular subspace, and so by Axiom \ref{AxiomII} and Lemma \ref{lem:polarProperties} \ref{lem:polarProperties-collinearPoints}, $M_i\setminus L_i,N_i\setminus L_i$ are contained in a unique maximal singular subspaces $M_i,N_i$ respectively. Then $\Aff{M_i},\Aff{N_i}$ are unique maximal cliques containing $\Aff[*]{M_i\setminus L_i},\Aff[*]{N_i\setminus L_i}$ respectively. This forces $\Aff{M_1}=u-w+\Aff{M_0}$ and $\Aff{N_1}=u-w+\Aff{N_0}$. But $\Aff{M_i},\Aff{N_i}$ are vector spaces, so we must have $\Aff{M_1}=\Aff{M_0},\Aff{N_1}=\Aff{N_0}$, $w-u\in \Aff{M_0}\cap \Aff{N_0} = \Aff{L_0}= \Aff{L_1}$ and $L_1=L_0$.
    
    Now suppose $L_0=L_1$ and $u-w\in \Aff{L_0}$. Then for all $M,N\in \Sigma_{L_0}$, $u-w\in \Aff{L_0}=\Aff{M}\cap \Aff{N}$, so $u-w+\Aff{L_0}=\Aff{L_0}$ and $u-w+\Aff[*]{M\setminus L_0}=\Aff{M\setminus L_0},u-w+\Aff[*]{N\setminus L_0}=\Aff[*]{N\setminus L_0}$. This shows that $u-w+V\Delta_{L_0}= V\Delta_{L_0}$, so we have $u+V\Delta_{L_0}= w + V\Delta_{L_0}$.    
\end{proof}

Now we can count the number of optimal $\theta_1$-eignefunctions of the hyperbolic affine polar graphs.

\begin{corollary}
    Let $V$ be a vector space over $\mathbb{F}_q$  of dimension $2m$, $m\geqslant 1$, and $Q$ be the nondegenerate quadratic form on $V$ of type $+1$. Then there are exactly 
    \begin{equation*}
        q^{m+1}\left(\frac{q^{2m}-1}{q-1}\right)\prod_{i=0}^{m-1}\left(q^{m-i-1}+1\right)
    \end{equation*}
    pairs $\{T_0,T_1\}$ of isolated cliques of size $\theta_1+1$    
\end{corollary}
\begin{proof}
    Let $N$ be the number such isolated cliques. By Proposition \ref{prop:affineCharacterisation}, these isolated cliques are exactly the elements of the sets $v+V\Delta_L$ for some singular subspace of $\Pi$ with $\dim(L)=n-2$. By Lemma \ref{lem:hyperbolicDistinct}, for coset representatives $v_i$ of $\Aff{L}$, with $i\in \{0,1,\dots,q^{m+1}-1\}$, we have a disjoint union of sets
    \begin{equation*}
        SV\Delta_L=\bigcup_{i=0}^{q^{m+1}-1} (v_i + V\Delta_L). 
    \end{equation*}
    Also by Corollary \ref{lem:hyperbolicDistinct}, $SV\Delta_L$ are disjoint as $L$ varies. As $\delta=|V\Delta_L|=|\Delta_L|=1(1+1)/2=1$ for all such $L$, we see that $N/(\delta q^{m+1})$ is the number of singular subspaces of $\Or[+]{2m}{q}$ of dimension $m-2$. The result follows by \cite[Lemma 9.4.1]{BCN89}.    
\end{proof}
~\\

\subsubsection{The $\theta_1$-eigenfunctions of elliptic affine polar graphs}\label{sssec:theta1_elliptic}

In this section we consider the affine polar graphs corresponding to a quadratic form of type $-1$ (the elliptic case). For this case, we will need some extra notation and basic results on quadratic forms. Note that the rank of $\Or[-]{2m}{q}$ is $m-1$, so maximal singular subspaces have dimension $m-2$.

Let $V$ be a vector space over $\mathbb{F}_q$ of dimension $2m$, $m\geqslant 1$, and $Q$ be the nondegenerate quadratic form on $V$ of type $-1$. Further let $B$ be the polarisation of $Q$, and for any $S\subseteq V$, define
\begin{equation*}
    S^{\perp}=\{u\in V:B(u,s)=0\text{ for all }s\in S\}.
\end{equation*}

\begin{lemma}\label{lem:affineSpaces}
    Let $V$ be a vector space over $\mathbb{F}_q$ of dimension $2m$, $m\geqslant 1$, and $Q$ be the nondegenerate quadratic form on $V$ of type $\epsilon$. For any subspace $U\subseteq V$ we have;
    \begin{enumerate}
        \item\label{lem:affineSpaces-dimperp} $U^{/perp}$ is a subapces, and $\dim(U^{\perp})=\dim(V)-\dim(U)$;
        \item\label{lem:affineSpaces-perpperp} $U=U^{\perp\perp}$;
    \end{enumerate}
    If $U\subseteq\Quadric{V}$, we have
    \begin{enumerate}\setcounter{enumi}{2}
        \item\label{lem:affineSpaces-quadricperp} $U\subseteq U^{\perp}$;
        \item\label{lem:affineSpaces-quadricneq} if $U$ is maximal in $\Quadric{V}$, $Q(t)\neq 0$ for all $t\in U^{\perp}\setminus U$.
    \end{enumerate}
\end{lemma}
\begin{proof}
    1. See \cite[Lemma 3.1]{B15}.

    2. See \cite[Theorem 3.4]{B15}

    3. See \cite[Lemma 3.19]{B15}

    4. Otherwise $U\subsetneq\lrangle{t,U}\subseteq\Quadric{V}$, conradicting maximality of $U$.
\end{proof}

In the previous sections, the fact that a maximal clique is regular was used to show our constructions defined a $\theta_1$-eigenfunction. However, this is not true in the elliptic affine polar graphs, and we have to study adjacency of vertices from outside of a maxmimal clique in more detail.

\begin{lemma}\label{lem:ellipticNeighbours}
    Let $V$ be a vector space over $\mathbb{F}_q$ of dimension $2m$, $m\geqslant 1$, and $Q$ be the nondegenerate quadratic form on $V$ of type $-1$. Further let $v\in V$ and $M$ be a maximal singular subspace of $\Or[-]{2m}{q}$. Then for any $z\in V\setminus (v+\Aff{M})$, the neighbours $V_z$ of $z$ in the graph $\VO[-]{2m}{q}$ are such that
    \begin{equation*}
        |V_z\cap (v+\Aff{M})|=
            \left\{
              \begin{array}{cc}
                q^{m-1} - 1 & z\in v+\Aff{M}\\
                q^{m-2}, & z\notin (v+\Aff{M}^{\perp}); \\
                0 & z\in (v+\Aff{M}^{\perp}\setminus \Aff{M})
              \end{array}
            \right.
    \end{equation*}
\end{lemma}
\begin{proof}
    As being a adjacent and nonadjacent is invariant under the action of an automorphism, we can assume $v=0$. By Lemma \ref{lem:affineCliques} \ref{lem:affineCliques-maximal}, $\Aff{M}$ is a maximal clique. For the remainder of the proof, we let $U=\Aff{M}$. 

    Suppose $z\in U^{\perp}\setminus U$. For any $u\in U$, $Q(z-u)=B(z,u)+Q(z)+Q(u)=Q(z)$ as $z\in U^{\perp}$ and $u\in \Quadric{V}$. By Lemma \ref{lem:affineSpaces} \ref{lem:affineSpaces-quadricneq}, $Q(z)\neq 0$ and $u$ is not adjacent to $z$. Therefore, $|V_z\cap U|=0$.

    Suppose $z\notin U^{\perp}$. Then there exists $u\in U$ such that $B(z,u)\neq 0$. The function $b_z:V\to \mathbb{F_q}$ defined by $b_z(v)=B(z,v)$ is a linear function with rank 1, and therefore $\dim(\ker(b_z))=\dim(V)-1=2m-1$. Then we see that 
    \begin{align*}
    \dim(U\cap \ker(b_z)) &=\dim(U)+\dim(\ker(b_z)) - \dim(U+\ker(b_z))\\
    &=3m-2-\dim(U+\ker(b_z))\geqslant m-2.
    \end{align*}
    But $\dim(U)=m-1$ and $u\in U\setminus \ker(b_z)$, so $\dim(U\cap \ker(b_z))\leqslant m-2$, and we have shown $\dim(U\cap \ker(b_z))= m-2$. Then we have $w\in U$ such that $Q(z-w)=0$  if and only if $0=b_z(w)+Q(z)+Q(w)=b_z(w)+Q(z)$, or $b_z(w)=-Q(z)$. But for any $x\in U$  such that $b_z(x)=-Q(z)$ (which exist because $b_z$ is nonzero on $U$), we have equality of sets $\{y\in U:b_z(y)=-Q(z)\}=x+\ker(b_z)$, and has size $|\ker(b_z)|=q^{m-2}$.   
\end{proof}

The above result gives a 3 distinct cases for the intersection of a neighbourhood of a vertex outside lying outside of a maximal clique with this clique. We will use this to construct optimal $\theta_1$-eigenfunctions.

\begin{lemma}\label{lem:ellipticConstruction}
    Let $V$ be a vector space over $\mathbb{F}_q$ of dimension $2m$, $m\geqslant 1$, and $Q$ be the nondegenerate quadratic form on $V$ of type $-1$. Further, let $v\in V$, $M$ be a maximal singular subspace of $\Or[-]{2m}{q}$ and $t\in \Aff{M}^{\perp}\setminus \Aff{M}$. Then in the graph $\VO[-]{2m}{q}$:
    \begin{enumerate}
        \item\label{lem:ellipticConstruction-adj} for all distinct $x,y\in v + \Aff{M}$ or $x,y\in t +  v + \Aff{M}$, $x$ and $y$ are adjacent;
        \item\label{lem:ellipticConstruction-notadj} for all $x\in v + \Aff{M},y\in t + v + \Aff{M}$, $x$ and $y$ are not adjacent;
        \item\label{lem:ellipticConstruction-func} the function $f:\mathcal{P}\to \mathbb{R}$, such that
        \begin{equation*}
            f(z) = 
                \left\{
                  \begin{array}{cc}
                    1, & z\in v + \Aff{M}; \\
                    -1, & z\in t + v + \Aff{M}; \\
                    0, & \text{otherwise}.
                  \end{array}
                \right.
        \end{equation*}
        satisfies condition  \eqref{LocalCondition} for $\theta_1=q^{m-1}-1$.
    \end{enumerate}
\end{lemma}
\begin{proof} 
    As being a adjacent, nonadjacent, and satisfying condition  \eqref{LocalCondition} is invariant under the action of an automorphism, we can assume $v=0$. Throughout, we let $U=\Aff{M}$.
    
    1. This follows from Axiom \eqref{AxiomI}.

    2. By Lemma \ref{lem:affineCliques} \ref{lem:affineCliques-maximal}, $U$ is a maximal subspace in $\Quadric{V}$. Then for all $u,w\in U$, $Q(u-(t+w))=Q((u-w)-t)=B(u-w,t)+Q(u-w)+Q(t)=Q(t)$, as $t\in U^{\perp}$ and $u-w\in U\subseteq \Quadric{V}$. But $Q(t)\neq 0$ by Lemma \ref{lem:affineSpaces} \ref{lem:affineSpaces-quadricneq}. 

    3. We have three cases for $z\in V$. The cases $z\in U$ and $z\in t + U$ can be verified using parts \ref{lem:ellipticConstruction-adj} and \ref{lem:ellipticConstruction-notadj}. The cases $z\in U^{\perp}\setminus ((t+U)\cup U)$ and
    $z\notin U^{\perp}$ follows from Lemma \ref{lem:ellipticNeighbours}, after noting that $t+U^{\perp}=U^{\perp}$.
\end{proof}

Now we show that any pair of isolated cliques of the sizes we are interested in come from the above construction.

\begin{proposition}
    Let $V$ be a vector space over $\mathbb{F}_q$ of dimension $2m$, $m\geqslant 1$, and $Q$ be the nondegenerate quadratic form on $V$ of type $-1$. For isolated cliques $T_0,T_1$ of size $\theta_1+1$ in $\VO[-]{2m}{q}$, there is a maximal singular subspace $M$ of $\Or[-]{2m}{q}$, $v\in V$ and $t\in\Aff{M}^{\perp}\setminus \Aff{M}$ such that $T_0=v+\Aff{M},T_1=t+v+\Aff{M}$.
\end{proposition}
\begin{proof}
    By Lemma \ref{lem:affineCliques} \ref{lem:affineCliques-maximal}, $T_0=v+\Aff{M}$ for some maximal singular subspace $M$. As being a adjacent and nonadjacent is invariant under the action of an automorphism, we can assume $v=0$. Throughout, we let $T_0=U=\Aff{M}$.

    By Lemma \ref{lem:affineCliques} \ref{lem:affineCliques-maximal} we also have $t\in V$ and maximal singular subspace $N$ such that $T_1=t+\Aff{N}$, and by Lemma \ref{lem:ellipticNeighbours}, $T_1\subseteq U^{\perp}$. In particular, $t\in U^{\perp}$ and $\Aff{N}\subseteq U^{\perp}\cap \Quadric{V}$. But by Lemma \ref{lem:affineSpaces} \ref{lem:affineSpaces-quadricneq}, $U^{\perp}\cap \Quadric{V}=U$, so $U=\Aff{N}$ and $M=N$. As $T_0,T_1$ are distinct, $t\in U^{\perp}\setminus U$.
\end{proof}

Finally, we count the number of such pairs of isolated cliques.

\begin{corollary}
    Let $V$ be a vector space over $\mathbb{F}_q$ of dimension $2m$, $m\geqslant 1$, and $Q$ be the nondegenerate quadratic form on $V$ of type $-1$. Then there are exactly
    \begin{equation*}
        q^{m-1}\binom{q^{m+1}}{2}\prod_{i=0}^{m-2}\left(q^{m-i}+1\right)
    \end{equation*}
    pairs $\{T_0,T_1\}$ of isolated cliques of size $\theta_1+1$
\end{corollary}
\begin{proof}
    For any such pair of isolated cliques $\{T_0,T_1\}$, there is a unique maximal singular subspace $M$ such that $T_0-T_1=t+\Aff{M}$, where $t\in \Aff{M}^{\perp}$. Therefore, any such pair corresponds to choosing two elements of the same coset of $\Aff{M}^{\perp}$. There are $q^{m-1}$ such cosets, and for each of these there are  $\binom{q^{m+1}}{2}$ choices of pairs of elements in the coset. The result follows from the number of maximal singular subspaces in \cite[Lemma 9.4.1]{BCN89}.
\end{proof}

\section{Tightness of WDB for the negative non-principal eigenvalue $\theta_2$ of unitary graphs $U(4,q)$}\label{TightnessU4q}
In this section we prove Proposition \ref{prop:Main2}.

Let $Q:=\{\delta \in \mathbb{F}_q^* \mid \delta^{\sqrt{q}+1} = 1\}$.
Note that for any $\gamma \in Q$, we have $\gamma^{\sqrt{q}} = 1/\gamma$.

By definition, for arbitrary non-zero isotropic vectors $v = (v_1,v_2,v_3,v_4)$ and $u = (u_1,u_2,u_3,u_4)$, the vertices 
$[v]$ and $[u]$ are adjacent in $U(4,q)$ if and only if
$$
v_1u_1^{\sqrt{q}}+v_2u_2^{\sqrt{q}}+v_3u_3^{\sqrt{q}}+v_4u_4^{\sqrt{q}} = 0.
$$

Consider the following two cases.

\medskip
\noindent
\textbf{Case 1:} $q$ is even.

Consider the following two subsets of points  
\begin{equation}\label{UT0ver1}
T_0:=\{[(1,\gamma,0,0)]\mid \gamma \in Q\},
\end{equation}
\begin{equation}\label{UT1ver1}
T_1:=\{[(0,0,1,\gamma)]\mid \gamma \in Q\}.
\end{equation}
of the skew projective lines 
$$
L_0:=\{[(1,\delta,0,0)]\mid \delta \in \mathbb{F}_q\} \cup \{[(0,1,0,0)]\},
$$ 
$$
L_1:=\{[(0,0,1,\delta)]\mid \delta \in \mathbb{F}_q\} \cup \{[(0,0,0,1)]\},
$$
respectively,
in $PG(3,q)$. Note that the points from $T_0 \cup T_1$ are the only isotropic points from $L_0 \cup L_1$. 
We also equivalently have
\begin{equation}\label{UT0ver2}
T_0:=\{[(\gamma,1,0,0)]\mid \gamma \in Q\},
\end{equation}
\begin{equation}\label{UT1ver2}
T_1:=\{[(0,0,\gamma,1)]\mid \gamma \in Q\}.
\end{equation}

Note that $|T_0| = |T_1| = \sqrt{q}+1$. Moreover, $T_0 \cup T_1$ induces a complete bipartite subgraph in $U(4,q)$ with parts $T_0$ and $T_1$.  

We show that every vertex $u$ of $U(4,q)$ that does not belong to $T_0 \cup T_1$ has at most one neighbour in $T_0$ and at most one neighbour in $T_1$. Moreover, we show that every vertex $u$ of $U(4,q)$ that does not belong to $T_0 \cup T_1$ has one neighbour in $T_0$ if and only if $u$ has one neighbour in $T_1$. Consider a vertex $u = [(u_1,u_2,u_3,u_4)] \notin T_0 \cup T_1$. The property $u \notin T_0$ implies $u_3 \ne 0$ or $u_4 \ne 0$. The property $u \notin T_1$ implies $u_1 \ne 0$ or $u_2 \ne 0$. We also note that, for any distinct $i,j \in \{1,2,3,4\}$, the property 
$u_i^{\sqrt{q}+1} = u_j^{\sqrt{q}+1}$ implies $u_k^{\sqrt{q}+1} = u_\ell^{\sqrt{q}+1}$, where $\{k,\ell\} = \{1,2,3,4\} \setminus \{i,j\}$. Indeed, it follows from the fact that $u$ is isotropic, that is, from the condition 
\begin{equation}\label{uIsIsotropic}
u_1^{\sqrt{q}+1} + u_2^{\sqrt{q}+1} + u_3^{\sqrt{q}+1} + u_4^{\sqrt{q}+1} = 0.    
\end{equation}

Consider the following four cases.

\medskip
\noindent
\textbf{Case 1.1:} $u_1 \ne 0$ and $u_3 \ne 0$. It is convenient to use expressions (\ref{UT0ver1}) and (\ref{UT1ver1}) here.
The vertex $u$ is adjacent to $[(1,\gamma,0,0)]$ if and only if 
$$u_1 + u_2\gamma^{\sqrt{q}} = 0,$$
or, equivalently,
$$\gamma = u_2/u_1.$$
Thus, the vertex $u$ has at most one neighbour in $T_0$. It has exactly one neighbour, namely, $[(1,u_2/u_1,0,0)]$ if and only if $u_2/u_1 \in Q$, that is, if and only if 

\begin{equation}\label{u1u2HaveTheSameNorm}
u_1^{\sqrt{q}+1} = u_2^{\sqrt{q}+1}.    
\end{equation}

In view of condition (\ref{uIsIsotropic}), condition (\ref{u1u2HaveTheSameNorm}) is equivalent to the following condition:
\begin{equation}\label{u3u4HaveTheSameNorm}
u_3^{\sqrt{q}+1} = u_4^{\sqrt{q}+1}.    
\end{equation}
 
The vertex $u$ is adjacent to $[(0,0,1,\gamma)]$ if and only if 
$$u_3 + u_4\gamma^{\sqrt{q}} = 0,$$
or, equivalently,
$$\gamma = u_4/u_3.$$
Thus, the vertex $u$ has at most one neighbour in $T_1$. It has exactly one neighbour, namely, $[(0,0,1,u_4/u_3)]$ if and only if $u_4/u_3 \in Q$, that is, if and only if condition \ref{u3u4HaveTheSameNorm} holds.

\medskip
\noindent
\textbf{Case 1.2:} $u_1 \ne 0$ and $u_4 \ne 0$. It is convenient to use expressions (\ref{UT0ver1}) and (\ref{UT1ver2}) here.
The proof is analogous to Case 1.1.

\medskip
\noindent
\textbf{Case 1.3:} $u_2 \ne 0$ and $u_3 \ne 0$. It is convenient to use expressions (\ref{UT0ver2}) and (\ref{UT1ver1}) here.
The proof is analogous to Case 1.1.

\medskip
\noindent
\textbf{Case 1.4:} $u_2 \ne 0$ and $u_4 \ne 0$. It is convenient to use expressions (\ref{UT0ver2}) and (\ref{UT1ver2}) here.
The proof is analogous to Case 1.1.

\medskip
\noindent
\textbf{Case 2:} $q$ is odd. 

Let $\beta$ be a primitive element in $\mathbb{F}_q$ and let $\varepsilon = \beta^{\frac{\sqrt{q}-1}{2}}$. Note that $\varepsilon^{\sqrt{q}+1} = -1$ and $\varepsilon^{\sqrt{q}} = -1/\varepsilon$. 

Consider the following two subsets of points  
\begin{equation}\label{UoddqT0ver1}
T_0:=\{[(1,\varepsilon\gamma,0,0)]\mid \gamma \in Q\},
\end{equation}
\begin{equation}\label{UoddqT1ver1}
T_1:=\{[(0,0,1,\varepsilon\gamma)]\mid \gamma \in Q\}.
\end{equation}
of the skew projective lines 
$$
L_0:=\{[(1,\delta,0,0)]\mid \delta \in \mathbb{F}_q\} \cup \{[(0,1,0,0)]\},,
$$ 
$$
L_1:=\{[(0,0,1,\delta)]\mid \delta \in \mathbb{F}_q\} \cup \{[(0,0,0,1)]\},
$$
respectively, in $PG(3,q)$.
Note that the points from $T_0 \cup T_1$ are the only isotropic points from $L_0 \cup L_1$. 
We also equivalently have
\begin{equation}\label{UoddqT0ver2}
T_0=\{[(\varepsilon\gamma,1,0,0)]\mid \gamma \in Q\},
\end{equation}
\begin{equation}\label{UoddqT1ver2}
T_1=\{[(0,0,\varepsilon\gamma,1)]\mid \gamma \in Q\}.
\end{equation}

Note that $|T_0| = |T_1| = \sqrt{q}+1$. Moreover, $T_0 \cup T_1$ induces a complete bipartite subgraph in $U(4,q)$ with parts $T_0$ and $T_1$.  

We show that every vertex $u$ of $U(4,q)$ that does not belong to $T_0 \cup T_1$ has exactly one neighbour in $T_0$ and exactly one neighbour in $T_1$. Consider a vertex $u = [(u_1,u_2,u_3,u_4)] \notin T_0 \cup T_1$. The property $u \notin T_0$ implies $u_3 \ne 0$ or $u_4 \ne 0$. The property $u \notin T_1$ implies $u_1 \ne 0$ or $u_2 \ne 0$.
We also note that, for any distinct $i,j \in \{1,2,3,4\}$, the property 
$u_i^{\sqrt{q}+1} = -u_j^{\sqrt{q}+1}$ implies $u_k^{\sqrt{q}+1} = -u_\ell^{\sqrt{q}+1}$, where $\{k,\ell\} = \{1,2,3,4\} \setminus \{i,j\}$. Indeed, it follows from the fact that $u$ is isotropic, that is, from condition (\ref{uIsIsotropic}).

Consider the following four cases.

\medskip
\noindent
\textbf{Case 2.1:} $u_1 \ne 0$ and $u_3 \ne 0$. It is convenient to use expressions (\ref{UoddqT0ver1}) and (\ref{UoddqT1ver1}) here.
The vertex $u$ is adjacent to $[(1,\varepsilon\gamma,0,0)]$ if and only if 
$$u_1 + u_2(\varepsilon\gamma)^{\sqrt{q}} = 0,$$
or, equivalently,
$$\varepsilon\gamma = u_2/u_1.$$
Thus, the vertex $u$ has at most one neighbour in $T_0$. It has exactly one neighbour, namely, $[(1,u_2/u_1,0,0)]$ if and only if $u_2/u_1 \in \varepsilon Q$, that is, if and only if 

\begin{equation}\label{u1u2HaveOppositeNorms}
u_1^{\sqrt{q}+1} = -u_2^{\sqrt{q}+1}.    
\end{equation}

In view of condition (\ref{uIsIsotropic}), condition (\ref{u1u2HaveOppositeNorms}) is equivalent to the following condition:
\begin{equation}\label{u3u4HaveOppositeNorms}
u_3^{\sqrt{q}+1} = -u_4^{\sqrt{q}+1}.    
\end{equation}
 
The vertex $u$ is adjacent to $[(0,0,1,\varepsilon\gamma)]$ if and only if 
$$u_3 + u_4(\varepsilon\gamma)^{\sqrt{q}} = 0,$$
or, equivalently,
$$\varepsilon\gamma = u_4/u_3.$$
Thus, the vertex $u$ has at most one neighbour in $T_1$. It has exactly one neighbour, namely, $[(0,0,1,u_4/u_3)]$ if and only if $u_4/u_3 \in \varepsilon Q$, that is, if and only if condition \ref{u3u4HaveOppositeNorms} holds.

\medskip
\noindent
\textbf{Case 2.2:} $u_1 \ne 0$ and $u_4 \ne 0$. It is convenient to use expressions (\ref{UoddqT0ver1}) and (\ref{UoddqT1ver2}) here.
The proof is analogous to Case 2.1.

\medskip
\noindent
\textbf{Case 2.3:} $u_2 \ne 0$ and $u_3 \ne 0$. It is convenient to use expressions (\ref{UoddqT0ver2}) and (\ref{UoddqT1ver1}) here.
The proof is analogous to Case 2.1.

\medskip
\noindent
\textbf{Case 2.4:} $u_2 \ne 0$ and $u_4 \ne 0$. It is convenient to use expressions (\ref{UoddqT0ver2}) and (\ref{UoddqT1ver2}) here.
The proof is analogous to Case 2.1.

\section*{Acknowledgments}
Rhys J. Evans was supported by the Slovenian Research and Innovation Agency (ARIS), research projects J1-4351 and P1-0294.
Sergey Goryainov is supported by the Special Project on Science and Technology Research and Development Platforms, Hebei Province (22567610H). Leonid Shalaginov is supported by Russian Science Foundation according to the research project 22-21-20018.

\end{document}